\documentclass[11pt, notitlepage]{article}      
\usepackage{amsmath,amssymb,amscd,amsthm, graphicx, verbatim, setspace} 
\usepackage{authblk}

\theoremstyle{plain}
\newtheorem{thm}{Theorem}[section]
\newtheorem{lem}[thm]{Lemma}
\newtheorem{cor}[thm]{Corollary}

\usepackage{mathtools}
\usepackage{amsthm}
\usepackage{amssymb}
\usepackage{amsfonts}
\usepackage{graphicx}
\usepackage{ytableau}
\usepackage{hyperref}
\usepackage[utf8]{inputenc}
\usepackage{todonotes}
\usepackage{tikz}
\usepackage{comment}
\DeclarePairedDelimiter{\ceil}{\lceil}{\rceil}

\newcommand{\ex}{\operatorname{ex}}
\newcommand{\Ex}{\operatorname{Ex}}
\newcommand{\F}{\operatorname{F}}
\newcommand{\up}{\operatorname{up}}
\newcommand{\fw}{\operatorname{fw}}
\newcommand{\fl}{\operatorname{fl}}
\newcommand{\Up}{\operatorname{Up}}

\DeclareMathOperator{\LCS}{LCS}
\begin{document}

\title{Formations and generalized Davenport-Schinzel sequences}

\author{Jesse Geneson \thanks{Dept.~of Mathematics and Statistics, San Jose State University, San Jose, CA, USA (jesse.geneson@sjsu.edu) } \and Peter Tian \thanks{Dept.~of Operations Research and Financial Engineering, Princeton University, Princeton, NJ, USA (ptian@princeton.edu)} \and Katherine Tung \thanks{Harvard University, Cambridge, MA, USA (katherinetung@college.harvard.edu)} 
}

\maketitle

\begin{abstract}
Let $\up(r, t) = (a_1 a_2 \dots a_r)^t$. We investigate the problem of determining the maximum possible integer $n(r, t)$ for which there exist $2t-1$ permutations $\pi_1, \pi_2, \dots, \pi_{2t-1}$ of $1, 2, \dots, n(r, t)$ such that the concatenated sequence $\pi_1 \pi_2 \dots \pi_{2t-1}$ has no subsequence isomorphic to $\up(r,t)$. This quantity has been used to obtain an upper bound on the maximum number of edges in $k$-quasiplanar graphs. It was proved by (Geneson, Prasad, and Tidor, Electronic Journal of Combinatorics, 2014) that $n(r, t) \le (r-1)^{2^{2t-2}}$.

We prove that $n(r,t) = \Theta(r^{2t-1 \choose t})$, where the constant in the bound depends only on $t$. Using our upper bound in the case $t = 2$, we also sharpen an upper bound of (Klazar, Integers, 2002), who proved that $\Ex(\up(r,2),n) < (2n+1)L$ where $L = \Ex(\up(r,2),K-1)+1$, $K = (r-1)^4 + 1$, and $\Ex(u, n)$ denotes the extremal function for forbidden generalized Davenport-Schinzel sequences.  We prove that $K = (r-1)^4 + 1$ in Klazar's bound can be replaced with $K = (r-1) \binom{r}{2}+1$. 

We also prove a conjecture from (Geneson, Prasad, and Tidor, Electronic Journal of Combinatorics, 2014) by showing for $t \geq 1$ that $\Ex(a b c (a c b)^{t} a b c, n) = n 2^{\frac{1}{t!}\alpha(n)^{t} \pm O(\alpha(n)^{t-1})}$. In addition, we prove that $\Ex(a b c a c b (a b c)^{t} a c b, n) = n 2^{\frac{1}{(t+1)!}\alpha(n)^{t+1} \pm O(\alpha(n)^{t})}$ for all $t \geq 1$.
\end{abstract}

\section{Introduction}

We say that a sequence $v$ \emph{contains} a sequence $u$ if $v$ has some subsequence $v'$ (not necessarily contiguous) that is isomorphic to $u$ ($v'$ can be changed into $u$ by a one-to-one renaming of its letters). Otherwise $v$ \emph{avoids} $u$. We call a sequence \emph{$r$-sparse} if every $r$ consecutive letters are distinct. \emph{Davenport-Schinzel sequences of order $s$} avoid alternations of length $s+2$ and have no adjacent same letters \cite{DS}. \emph{Generalized Davenport-Schinzel sequences} avoid a forbidden sequence $u$ (or a family of sequences) and are $r$-sparse, where $r$ is the number of distinct letters in $u$. 

For any sequence $u$, define $\Ex(u, n)$ to be the maximum possible length of an $r$-sparse sequence with $n$ distinct letters that avoids $u$, where $r$ is the number of distinct letters in $u$. Furthermore, define $\Ex(u, n, m)$ to be the maximum possible length of a sequence with $n$ distinct letters that avoids $u$ and can be partitioned into $m$ contiguous blocks of distinct letters. Applications of $\Ex(u, n)$ include upper bounds on the complexity of lower envelopes of sets of polynomials of bounded degree \cite{DS}, the complexity of faces in arrangements of arcs with bounded pairwise crossings \cite{agsh}, and the maximum number of edges in $k$-quasiplanar graphs \cite{foxpachsuk}. The function $\Ex(u, n,m)$ has been used to find bounds on $\Ex(u, n)$.

Bounds on $\Ex(u, n)$ are known for several families of sequences such as alternations \cite{agshsh,niv, pettie} and more generally the sequences $\up(r, t) = (a_1 a_2 \dots a_r)^t$ \cite{gpt}. Let $a_s$ denote the alternation of length $s$. It is known that that $\Ex(a_3, n) = n$, $\Ex(a_4, n) = 2n-1$, $\Ex(a_5, n) = 2n \alpha(n) + O(n)$, $\Ex(a_6, n) = \Theta(n 2^{\alpha(n)})$, $\Ex(a_7, n) = \Theta(n \alpha(n) 2^{\alpha(n)})$, and $\Ex(a_{s+2}, n) = n 2^{\frac{\alpha^{t}(n)}{t!}  \pm O(\alpha(n)^{t-1})}$ for all $s \geq 6$, where $t = \lfloor \frac{s-2}{2} \rfloor$ \cite{DS,agshsh,niv,pettie}. 

Relatively little about $\Ex(u, n)$ is known for arbitrary forbidden sequences $u$. However, one way to find upper bounds on $\Ex(u, n)$ for any sequence $u$ is to use \emph{$(r, s)$-formations}, which are concatenations of $s$ permutations of $r$ distinct letters. We define $\mathcal{F}_{r, s}$ to be the family of all $(r, s)$-formations. We define the function $\F_{r, s}(n)$ to be the maximum possible length of an $r$-sparse sequence with $n$ distinct letters that avoids all $(r, s)$-formations, and we define the function $\F_{r, s}(n, m)$ to be the maximum possible length of a sequence with $n$ distinct letters that avoids all $(r, s)$-formations and can be partitioned into $m$ blocks of distinct letters. Like $\Ex(u, n, m)$ and $\Ex(u, n)$, the function $\F_{r, s}(n, m)$ has been used to find bounds on $\F_{r, s}(n)$.

Let the formation width $\fw(u)$ denote the minimum $s$ for which there exists $r$ such that every $(r, s)$-formation contains $u$, and let the formation length $\fl(u)$ denote the minimum value of $r$ for which every $(r, \fw(u))$-formation contains $u$. These parameters were defined in \cite{gpt}, where it was observed that $\Ex(u, n) = O(\F_{\fl(u), \fw(u)}(n))$. This uses the fact that increasing the sparsity in the definition of $\Ex(u, n)$ only changes the value by at most a constant factor, which was proved by Klazar in \cite{klazar1}. Using the upper bound with $\fw(u)$ and known bounds on $\F_{r, s}(n)$, it is possible to find sharp bounds on $\Ex(u, n)$ for many sequences $u$.

In \cite{gpt}, Geneson, Prasad, and Tidor proved that $\fw(\up(r, t)) = 2t-1$ and $\fl(\up(r,t)) \le (r-1)^{2^{2t-2}}+1$. This implies $\Ex(\up(r, t), n) = n 2^{\frac{1}{(t-2)!}\alpha(n)^{t-2} \pm O(\alpha(n)^{t-3})}$ for every $r \ge 2$ and $t \ge 3$, where the constants in the bounds depend on $r$. They used this to sharpen the upper bound from \cite{foxpachsuk} on the maximum number of edges in $k$-quasiplanar graphs where no pair of edges intersect in more than $O(1)$ points.

They also proved that $\fw(u) = 4$ and $\Ex(u, n) = \Theta(n \alpha(n))$ for any sequence $u$ of the form $a v a v' a$ such that $a$ is a letter, $v$ is a nonempty sequence of distinct letters excluding $a$, and $v'$ is obtained from $v$ by only shifting the first letter of $v$. Based on computing $\fw(abc(acb)^t a b c)$ for small values of $t$, they conjectured in \cite{gpt} that $\fw(abc(acb)^t a b c) = 2t+3$ for all $t \geq 0$ and that $\Ex(a b c (a c b)^{t} a b c, n) = n 2^{\frac{1}{t!}\alpha(n)^{t} \pm O(\alpha(n)^{t-1})}$ for $t \geq 1$. We affirm this conjecture, and we also prove that $\fw(abcacb(abc)^{t}acb) =2t+5$ and $\Ex(a b c a c b (a b c)^{t} a c b, n) = n 2^{\frac{1}{(t+1)!}\alpha(n)^{t+1} \pm O(\alpha(n)^{t})}$ for $t \geq 1$. In addition, we improve an upper bound of Klazar \cite{klazar}, who proved that $\Ex(\up(r, 2), n) < (2n+1)L$, with $L = \Ex(\up(r, 2),K-1)+1$ and $K = (r-1)^4 + 1$. Here we prove that $K = (r-1)^4 + 1$ in Klazar's bound can be replaced with $K = (r-1)\binom{r}{2}+1$. 

We obtain this bound by proving that every $((r-1)\binom{r}{2}+1,3)$-formation contains $\up(r, 2)$, using a result about strongly unimodal sequences. On the other hand, we also prove that this result is sharp up to a constant factor. Specifically, we prove that there exist $(m, 3)$-formations with $m = \Omega(r^3)$ which avoid $\up(r, 2)$. As a result, we have $\fl(\up(r,2)) = \Theta(r^3)$. A similar result was also proved in \cite{beame}, but our result is better by a constant factor. See Section~\ref{upr2n}.

More generally, we prove that $\fl(\up(r,t)) = \Theta(r^{2t-1 \choose t})$, where the constants in the bound depend only on $t$. This improves the upper bound on $\fl(\up(r,t))$ from \cite{gpt}, and we prove a lower bound that matches the upper bound up to a constant factor that depends only on $t$. Using Klazar's sparsity lemma from \cite{klazar1}, our upper bound on $\fl(\up(r,t))$ also implies for all $n, r, t \ge 1$ that
$$\Ex(\up(r, t),n) \le (1+\Ex(\up(r,t),(r-1)^{{2t-1 \choose t}}))\F_{(r-1)^{{2t-1 \choose t}}+1, 2t-1}(n).$$ 

These new results are proved in Section \ref{uprtupper}.

In addition to using formations to obtain upper bounds on $\Ex(\up(r, t),n)$, we also use formations in Section~\ref{fws} to bound the extremal functions of other forbidden sequences. We show that $\Ex(a b c (a c b)^{t} a b c, n) = n 2^{\frac{1}{t!}\alpha(n)^{t} \pm O(\alpha(n)^{t-1})}$ and $\Ex(a b c a c b (a b c)^{t} a c b, n) = n 2^{\frac{1}{(t+1)!}\alpha(n)^{t+1} \pm O(\alpha(n)^{t})}$ using formation width. 

In Section \ref{exact}, we investigate subsequences $u$ of $\up(r,2)$ for which the exact values of $\Ex(u,n)$ and $\Ex(u,n,m)$ were not previously known. We find the exact values of $\Ex(\up(r,1) a_x, n)$ and $\Ex(\up(r,1) a_x, n, m)$ for $x \in \left\{1, \dots, r\right\}$. We also determine the exact values of $\F_{r, 2}(n)$, $\F_{r, 3}(n)$, $\F_{r, 2}(n, m)$, and $\F_{r, 3}(n, m)$. In Section \ref{ddim}, we extend the exact results about formations in sequences from Section~\ref{exact} to exact results about formations in $d$-dimensional 0-1 matrices.

\section{Definitions}

A \textit{restricted} $up(r,2)$ is an $\up(r,2)$ completely contained within any two permutations among $\{\pi_1, \pi_2, \pi_3\}$ in a formation $[\pi_1, \pi_2, \pi_3]$. For example, $[12345, 15432, 32514]$ has a restricted $\up(2,2)$ of ${(24)}^2$ in $[\pi_1, \pi_3]$ and a (non-restricted) $\up(3,2)$ of ${(514)^2}$.

Generalizing the definition of a restricted $\up(r,2)$, a restricted $\up(r,t)$ is an $\up(r,t)$ completely contained within any $t$ permutations among $\{\pi_1, \pi_2, \pi_{2t - 1}\}$ in a formation $[\pi_1, \pi_2, ..., \pi_{2t - 1}]$. We denote a restricted $\up(r,t)$ as $\Up(r,t)$.

Whether a formation contains $\up(r,t)$ may depend on the order of the permutations of the formation, but whether a formation contains $\Up(r,t)$ is invariant under reordering the permutations.

A permutation of a set $S$ is said to have \textit{length} $|S|$, i.e., the length of a permutation is the number of characters in the permutation. Similarly, the length of a subsequence of a permutation is the number of characters in the subsequence. Let $\LCS(\pi_i,\pi_j)$ be the length of the longest common subsequence of $\pi_i$ and $\pi_j$. Note that $\LCS(\pi_i, \pi_j)$ is the maximum $m$ such that a restricted $\up(m,2)$ configuration is present in $\pi_i$ and $\pi_j$ within the formation $[\pi_1, \pi_2, \pi_3]$.

Let $\pi$ be a permutation of a totally ordered set $S$ and $\sigma$ a permutation of a totally ordered set $T$. We define a permutation $\pi \otimes \sigma$ of the Cartesian product $S \times T$ with the lexicographic ordering by $\pi \otimes \sigma (x,y) = (\pi(x),\sigma(y))$. 

A subsequence $a_1, a_2, ... a_t$ of a permutation $\pi$ is called \emph{strongly unimodal} if it is increasing or decreasing or for some $k \in \{2, ..., t - 1\}$, $$a_1 < a_2 < \ldots < a_k > a_{k+1} > \ldots > a_t. $$  

If $a,b$ are nonnegative integers, let $a \oplus b$ be their nim-sum, i.e., the integer whose binary expansion is the bitwise-XOR of the binary expansions of $a$ and $b$. So $4 \oplus 6 = 2.$ 

Suppose $n=2^k$ and $0\le m<n$. Define $\tau_m$ to be the involution in $S_{\{0,1,2...,n-1\}}$ so $\tau_m(a) = m \oplus a.$

For example, if $n=8$ then $\tau_4 = \tau_{100_2} = [45670123]$ in table form and $\tau_6 = [67452301]$. 

These can be used to construct formations with no $\Up(r,t)$.

\section{Bounds for $\up(r, 2)$}\label{upr2n}

In this section, we show that $\fl(\up(r,2)) = \Theta(r^3)$.

Klazar's proof that $\Ex(\up(r, 2), n) < (2n+1)L$, where $L = \Ex(\up(r, 2),K-1)+1$ and $K = (r-1)^4 + 1$, uses the Erd\H{o}s-Szekeres theorem to find the copy of $\up(r, 2)$ \cite{klazar}. We sharpen the upper bound on $\Ex(\up(r, 2), n)$ by proving that every $((r-1)\binom{r}{2}+1,3)$-formation contains $\up(r, 2)$. We also show that this containment result is best possible up to a constant factor by proving that there exist $(m, 3)$-formations with $m = \Omega(r^3)$ which avoid $\up(r, 2)$. 

The following result is mentioned in \cite{chung} as an unpublished result of Steele and Chv\'atal. The proof is not explicitly given in \cite{chung} so we supply one here.

\begin{thm}\cite{chung}
\label{unimodal}
Any permutation of length ${t \choose 2}+1$ contains a strongly unimodal sequence of length $t$.
\end{thm}
\begin{proof}
For any index $i$, let $x(t)$ be the length of the longest increasing subsequence ending in position $i$. Let $y(t)$ be the length of the longest decreasing subsequence starting in position $i$. There is a strongly unimodal sequence of length $\max_{i} x(i)+y(i)-1$, and the map $i \mapsto (x(i),y(i))$ is injective. There are only $t \choose 2$ possible images with $x+y-1 < t$.

\end{proof}
Theorem \ref{unimodal} is sharp. There are permutations of $\{1, 2, ..., {t \choose 2}\}$ with no unimodal sequence of length $t$. See below for a general formula and Figure 1 for an example with $t = 4$.

$${t \choose 2}\hspace{10px}{t \choose 2}-2\hspace{10px} {t \choose 2}-1\hspace{10px} {t \choose 2}-5\hspace{10px} {t \choose 2}-4\hspace{10px} {t \choose 2}-3\hspace{10px} ...\hspace{10px} 1\hspace{10px} 2\hspace{10px} 3\hspace{10px} ...\hspace{10px} t-1$$

\begin{figure}\label{no-strong-unimodal}
\centering
{\begin{tikzpicture}

\foreach \point in {(5,2),(4,1),(3,0)}{
    \fill \point circle (2pt);
}
\foreach \point in {(2,4),(1,3)}{
    \fill \point circle (2pt);
}
\foreach \point in {(0,5)}{
    \fill \point circle (2pt);
}
\end{tikzpicture}}
\caption{A geometric representation of the permutation $645123$, which is a permutation of length ${t \choose 2}$ with no strongly unimodal sequence of length $t$ for $t = 4$.}
\label{wedge}
\end{figure}

\begin{thm}\label{upperformupr2}

If $n = {r \choose 2}(r - 1) + 1$ then any $(n,3)$-formation contains an $\up(r,2)$.
\end{thm}
\begin{proof}
Let the formation be $[\pi_1,\pi_2,\pi_3]$. Without loss of generality, let $\pi_1=e$. By the Erd\H{o}s-Szekeres theorem, $\pi_2$ either contains an increasing subsequence of length $r$ or a decreasing subsequence of length ${r \choose 2}+1$. 
In the first case, $[\pi_1, \pi_2]$ contains an $\up(r,2)$ on the symbols of the increasing subsequence. 
In the second case, consider the ${r \choose 2}+1$ symbols of the decreasing subsequence in $\pi_3$. By Theorem \ref{unimodal}, this permutation contained in $\pi_3$ has a strongly unimodal subsequence of length $r$, in positions $a_1 < a_2 < \ldots < a_r$ with $\pi_3(a_1) < \ldots < \pi_3(a_k) > \pi_3(a_{k+1})>\ldots>\pi_3(a_r)$. Then the values $\pi_3(a_1)<\pi_3(a_2)< ... <\pi_3(a_k)$ in $\pi_1$ and the values $\pi_3(a_{k+1})> ... > \pi_3(a_r)$ in $\pi_2$ form a pattern in $[\pi_1,\pi_2]$ repeated in $\pi_3$, hence $[\pi_1,\pi_2,\pi_3]$ contains an $\up(r,2)$.

\end{proof}

\begin{cor}\label{flupr2upper}
For all $r$, we have $\fl(\up(r,2)) \le {r \choose 2}(r - 1) + 1$.
\end{cor}

Besides the use of Theorem \ref{upperformupr2}, the proof of the next theorem is the same as the proof of Klazar's bound in \cite{klazar}. 

\begin{thm}\label{vr}
$\Ex(\up(r, 2), n) < (2n+1)L$, where $L = \Ex(\up(r, 2),K-1)+1$ and $K = (r-1)\binom{r}{2} + 1$. 
\end{thm}

\begin{proof}
Let $u$ be an $r$-sparse sequence with at most $n$ distinct letters. Suppose that $u$ has length at least $(2n+1)L$. Split $u$ into $2n+1$ disjoint intervals, each of length at least $L$. At least one interval $I$ contains no first or last occurrence of any letter in $u$. If $I$ has fewer than $K$ distinct letters, then $I$ contains $\up(r, 2)$ by the definition of $I$ and $L$. If $I$ has at least $K$ distinct letters, then all of these letters occur before $I$, in $I$, and after $I$. Thus $u$ contains an $((r-1)\binom{r}{2} + 1, 3)$-formation. By Theorem \ref{upperformupr2}, $u$ contains $\up(r, 2)$, completing the proof.
\end{proof}

In the remainder of this section, we prove that the bound in Corollary \ref{flupr2upper} is sharp up to a constant factor. In order to prove this, we first show that any $\up(r, 2)$ in an $(n, 3)$-formation must contain a restricted $\up(\lceil r/3 \rceil,2).$

\begin{lem}\label{upr2lemrestrict}
Any $(n,3)$-formation containing an $\up(r,2)$ contains a restricted $\up(\lceil r/3 \rceil,2).$
\end{lem}
\begin{proof}
The $\up(r,2)$ can be factored into six possibly empty words $w_1w_2w_3w_1w_2w_3$ so that $\pi_1$ contains $w_1w_2$, $\pi_2$ contains $w_3w_1$, and $\pi_3$ contains $w_2w_3$. The longest of $w_1,w_2,$ and $w_3$ contains the repeated sequence of a restricted $\up(\lceil r/3 \rceil,2),$ e.g., if $w_2$ is the longest, then there is a restricted $\up(|w_2|,2)$ contained in $[\pi_1,\pi_3]$ and $|w_2| \ge \lceil r/3 \rceil.$
\end{proof}

Next, we prove a lemma about longest common subsequences of permutations of Cartesian products which we will use with a product construction to prove the lower bound on $\fl(\up(r, 2))$.

\begin{lem}
\label{LCS}
Let $\pi_1$ and $\pi_2$ be permutations of the same size and $\sigma_1$ and $\sigma_2$ be permutations of the same size. Then $\LCS(\pi_1\otimes \sigma_1,\pi_2 \otimes \sigma_2) = \LCS(\pi_1,\pi_2) \LCS(\sigma_1,\sigma_2).$ 

\end{lem}

\begin{proof}
Let $n = \LCS(\pi_1,\pi_2), m=\LCS(\sigma_1,\sigma_2).$
Given a common subsequence $a_1,a_2,\ldots,a_n$ of $\pi_1$ and $\pi_2$, and a common subsequence $b_1,b_2,\ldots,b_m$ of $\sigma_1$ and $\sigma_2$, then a common subsequence of $\pi_1 \otimes \sigma_1$ and $\pi_2 \otimes \sigma_2$ is $$(a_1,b_1),(a_1,b_2), \ldots ,(a_1,b_m),(a_2,b_1) \ldots (a_2,b_m), \ldots (a_n,b_m).$$

Suppose we have a common subsequence of length $mn+1$, say $$(c_{1}, d_{1}), (c_{2}, d_{2}), \ldots, (c_{mn+1}, d_{mn+1}).$$ Suppose for $h=1,2,$ the locations of the sequence are $$(\ell_{h,1},k_{h,1}) < (\ell_{h,2},k_{h,2}) < \ldots < (\ell_{h,mn+1},k_{h,mn+1}) $$ in $\pi_h \otimes \sigma_h$, so $\pi_h \otimes \sigma_h (\ell_{h,i},k_{h,i}) = (c_i,d_i).$

By the lexicographic ordering, $\ell_{h,1} \le \ell_{h,2} \le \ldots \le \ell_{h,mn+1}$.
Since for all $h$ and $i$, $\pi_h(\ell_{h,i}) = c_i,$ repetitions of first coordinates of values occur precisely when the first coordinates of locations are repeated: $\ell_{h,i} = \ell_{h,j} \iff c_i = c_j.$ Since the first coordinates of locations are weakly increasing, repetitions are adjacent, so repeated values in the sequence $(c_i)$ are adjacent. The distinct elements of the sequences $(\ell_{1,i})$ and $(\ell_{2,i})$ are the locations of a common subsequence of $\pi_1$ and $\pi_2$. Since $\LCS(\pi_1,\pi_2)=n$, there can be at most $n$ distinct elements among these $mn+1$. By the pigeonhole principle, the sequence $c_{1}, c_{2}, \ldots, c_{{mn+1}}$ contains at least $m + 1$ repetitions of some value, which must be adjacent, say $c_t= \ldots =c_{t+m}$. Then the common subsequence contains $(c_t, d_t), (c_{t}, d_{t+1}), \ldots, (c_{t}, d_{t+m})$. Then $d_t, d_{t+1}, \ldots, d_{t+m}$ is a common subsequence of $\sigma_1$ and $\sigma_2$ of length $m+1$, contradicting the assumption that $\LCS(\sigma_1,\sigma_2)=m$.  
\end{proof}

We provide a product construction in the next lemma, which we will use to prove the claimed lower bound of $\fl(\up(r,2)) = \Omega(r^3).$

\begin{lem}
\label{product}
If we have permutations $\pi_1, \pi_2, \pi_3 \in S_k$ so $LCS(\pi_i,\pi_j) \le r$ for each $(i,j) \in \{(1,2), (1,3), (2,3)\}$, then the $(k^t,3)$-formation $[\pi_1^{\otimes t},\pi_2^{\otimes t},\pi_3^{\otimes t}]$ contains no restricted $\up(r^t+1,2)$, hence no $\up(3r^t+1,2)$. 

\end{lem}
\begin{proof}
The longest common subsequences of $\pi_i^{\otimes t}$ and $\pi_j^{\otimes t}$ have length at most $r^t$ so there is no restricted $\up(r^t+1,2)$ in any pair, hence not in the $(k^t,3)$-formation. An $\up(3 r^t+1,2)$ would imply there is a restricted $\up(r^t+1,2)$ so there are no $\up(3 r^t+1,2)$s.
\end{proof}

Finally, we are ready to show that the construction in the last lemma gives the desired lower bound on $\fl(\up(r,2))$.

\begin{thm}
\label{Omega}
There are $(\Omega(r^3),3)$-formations with no $\up(r,2).$
\end{thm}
\begin{proof}
Let $\pi_1 = e, \pi_2 = 43218765, \pi_3 = 65872143$. Then for each $(i,j) \in \{(1,2), (1,3), (2,3)\}$, $\LCS(\pi_i,\pi_j) = 2$. By Theorem \ref{product}, for each $t$, there are $(8^t,3)$-formations with no restricted $\up(2^t+1,2)$ hence no $\up(3\cdot 2^t+1,2)$. 

We can use these constructions for powers of two to build examples of cubic size when $r$ is not a power of $2$.
For any $r\ge 4$, there is a power of two $2^t$ in $\big((r-1)/6,(r-1)/3\big]$. Then there is a $(2^{3t},3)$-formation with no restricted $\up(2^t+1,2)$ hence no $\up(3 \cdot 2^t+1,2)$. Since $3 \cdot 2^t+1 \le r$ and $\lceil \frac{(r-1)^3}{216} \rceil \le 2^{3t},$ there is an $(\lceil \frac{(r-1)^3}{216} \rceil ,3)$-formation with no $\up(r,2).$

\end{proof}

\begin{cor}
$\fl(\up(r,2)) = \Theta(r^3)$.
\end{cor}

\section{Bounds for $\up(r,t)$}\label{uprtupper}

In this section, we show that for any fixed $t$, $\fl(\up(r, t))$ is $\theta(r^g)$ where $g={2t-1 \choose t},$ extending the result for $t=2.$ 

We start by proving an upper bound on $\fl(\up(r, t))$. Since any formation which contains $\Up(r,t)$ must also contain $\up(r,t)$, we focus on $\Up(r, t)$ instead of $\up(r,t)$ for the upper bound.

\begin{thm} 
\label{upper}
In any $((r-1)^{g}+1,2t-1)$ formation where $g = {2t-1 \choose t}$, there is an $\Up(r,t).$
\end{thm}

For $i,j \in \{0,...,t\}$ let $g_t(i,j) = {2t-i-j \choose t-i}$. This counts the number of lattice paths from $(i,j)$ to $(t,t)$ or the number of ways a best-of-$(2t-1)$ match can end starting from a score of $(i,j).$ If $\max(i,j) <t$ then $g_t(i,j) = g_t(i+1,j) + g_t(i,j+1).$

\begin{thm}
\label{gthm}
Let $i,j \in \{0,...,t\}.$ In any $((r-1)^{g_t(i,j)}+1,2t-1)$ formation starting with $i$ identity permutations and $j$ copies of $w = [n~(n-1)~...~2~1]$, there is an $\Up(r,t)$. 
\end{thm}
\begin{proof}
Induct backwards on $i + j$. The base case is when $\max(i, j) = t$, so $g_t(i,j)=1$ and the statement is trivially true here. 

Suppose it is true for larger values of $i + j$. Consider a $\left((r-1)^{g_t(i,j)}+1,2t-1\right)$ formation starting with $e^i w^j$. The size of each permutation is $pq+1$ where $p=(r-1)^{g_t(i+1,j)}$ and $q=(r-1)^{g_t(i,j+1)}$. Apply the Erd\H{o}s-Szekeres theorem to the next permutation after $e^i w^j$. In a permutation of size $pq+1=(r - 1)^{g_t(i,j)} + 1$, there is either an increasing subsequence of length $p+1 = (r - 1)^{g_t(i + 1, j)} + 1$ or a decreasing subsequence of length $q+1=(r - 1)^{g_t(i, j + 1)} + 1$. Restricting to the symbols of this monotone subsequence gives us one more identity or $w$ permutation. By the inductive hypothesis, there must be an $\Up(r,t)$ using just those symbols.
\end{proof}

\begin{proof}[Proof of Theorem \ref{upper}] 
Theorem \ref{gthm} with $i=0, j=0$ says that an unrestricted $((r-1)^{2t \choose t}+1,2t-1)$ formation has an $\Up(r,t)$ whose pattern is increasing or decreasing. However, if we don't restrict the pattern to be monotone, we can relabel the symbols so that the first permutation is the identity. So, any $\left((r-1)^{g_t(1,0)}+1,2t-1\right) = \left((r-1)^{2t-1 \choose t}+1,2t-1\right)$ formation contains an $\Up(r,t)$. 
\end{proof}

\begin{cor}\label{fluprtupper}
For all $r, t \ge 1$, we have $\fl(\up(r,t)) \le (r-1)^{{2t-1 \choose t}}+1$.
\end{cor}

By the sparsity lemma of Klazar in \cite{klazar1}, we obtain the following upper bound on $\Ex(\up(r, t),n)$.

\begin{thm}
Let $n, r, t \ge 1$ and $g = {2t-1 \choose t}$. We have 
$$\Ex(\up(r, t),n) \le (1+\Ex(\up(r,t),(r-1)^{g}))\F_{(r-1)^{g}+1, 2t-1}(n).$$
\end{thm}

In the remainder of this section, we prove that the bound in Corollary \ref{fluprtupper} is sharp up to a constant factor. In order to prove this, we first generalize the result in Lemma \ref{upr2lemrestrict}.

\begin{thm}
\label{upUp}
If there is an $\up(r,t)$ in an $(n,2t-1)$-formation, then there is an $\Up(\lceil \frac{r}{t(t-1)+1}\rceil,t)$ in the formation.
\end{thm}
\begin{proof}
For each symbol in the $\up(r,t)$, there is a subset of size $t$ of the $2t-1$ indices of the permutations containing the corresponding symbols in the $\up(r,t).$ These form a chain in the lattice of subsets of $\{1,2,\ldots,2t-1\}$ of size $t,$ where the longest chain has $t(t-1)+1$ elements. By the pigeonhole principle, at least one of the $t$-element subsets is repeated at least $\frac{r}{t(t-1)+1}$ times, which means there is an $\Up(\lceil \frac{r}{t(t-1)+1}\rceil,t).$
\end{proof}

Let $p,q$ be nonnegative integers. We will think of them in base $2$. Suppose that the value $a$ appears in $\tau_{p}$ and $\tau_{q}$, and suppose that the value $b$ appears in $\tau_{p}$ and $\tau_{q}$, with $a < b$. Also, let $i$ be the leftmost (most significant) index of a binary digit where $a$ and $b$ disagree in binary.

\begin{thm} \label{binaryAgree}
If $a$ appears before $b$ in both $\tau_{p}$ and $\tau_{q}$, or if $b$ appears before $a$ in both $\tau_{p}$ and $\tau_{q}$, then $p$ and $q$ agree in binary position $i$.
\end{thm}
\begin{proof}
We prove the contrapositive: if $p$ and $q$ disagree in binary position $i$, then $a$ and $b$ appear in different orders in $\tau_{p}$ and $\tau_{q}$. Since $a < b$, and $a$ and $b$ first differ in binary position $i$, $a$ has a $0$ in that position while $b$ has a $1$ there. Without loss of generality, suppose $p$ has a $0$ in the $i$th position while $q$ has a $1$ there. Then $p \oplus a < p \oplus b$ while $q \oplus a > q \oplus b.$ Hence, $a$ and $b$ occur in different orders in $\tau_p$ and $\tau_q$.
\end{proof}

Next we show that the length of the longest common subsequence of $\{\tau_{p_1}, ..., \tau_{p_k}\}$ can be determined from the number of bits where the binary expansions of $p_1,...,p_k$ all agree.

\begin{thm}\label{agree_binary}
If $s$ is the number of bits where the binary expansions of $p_1,...,p_k$ all agree, then the longest common subsequence of $\{\tau_{p_1}, ..., \tau_{p_k}\}$ has length $2^s$. 
\end{thm}
\begin{proof}
This follows from the pigeonhole principle. More than $2^s$ elements of a subsequence would mean some pair $a$ and $b$ would have to agree in all $s$ bits where $p_1, ..., p_k$ agree, so the highest bit where $a$ and $b$ disagree would have to be one where $p_1,...,p_k$ are not unanimous. By Theorem \ref{binaryAgree}, $a$ and $b$ would appear in a different order in some $p_i$ from some $p_j$, so they could not be in a common subsequence.
\end{proof}

For example, $\tau_{0110_2}, \tau_{0011_2}, \tau_{0010_2}$ share a common subsequence of length $2^2=4$ since $0110$, $0011$, and $0010$ agree in two positions $0*1*$.

$\tau_{0110_2} = 6~7~4~5~2~3~0~1~14~15~12~13~10~11~8~9$

$\tau_{0011_2} = 3~2~1~0~7~6~5~4~11~10~9~8~15~14~13~12$

$\tau_{0010_2} = 2~3~0~1~6~7~4~5~10~11~8~9~14~15~12~13$

Common subsequences of length $4$ include $6~4~10~8$ and $7~5~11~8$.

\begin{thm}
\label{tensor}
Let $(\pi_1, \ldots, \pi_u)$ be a formation containing an $\Up(r,t)$ but no $\Up(r+1,t).$ Similarly, let $(\sigma_1, \ldots, \sigma_u)$ be a formation containing an $\Up(s,t)$ but no $\Up(s+1,t).$ Then $(\pi_1 \otimes \sigma_1, \ldots, \pi_u \otimes \sigma_u)$ contains an $\Up(rs,t)$ but no $\Up(rs+1,t)).$ 

\end{thm}
\begin{proof}
It suffices to consider $u=t$. The proof is analogous to that of Lemma \ref{LCS}, which considers the case $t = 2$.
\end{proof}

The construction in the following result is suboptimal but potentially still of interest.

\begin{thm}
There are $(\Omega(r^7),7)$ formations hence also $(\Omega(r^7),5)$ formations with no $\up(r,3).$
\end{thm}
\begin{proof}
The lines of a Fano plane have the property that any $3$ either intersect in a point and cover all points, or they all miss a single point. Thus, $\tau_{1101000_2}$, $\tau_{0110100_2}$, $\tau_{0011010_2}$, $\tau_{0001101_2}$, $\tau_{1000110_2}$, $\tau_{0100011_2}$, $\tau_{1010001_2}$ have the property that any triple has longest common subsequence $2^1=2$. 
By theorem \ref{tensor}, there is a $(128^k,7)$ formation with no $\Up(2^k+1,3).$ 
\end{proof}

The bound in the next theorem is sharp and better than the Fano construction since it, for example, yields a $\left(2^{10},5\right)$ formation with no $\Up(3,3).$

\begin{thm}
\label{uprtconstruction}
For all fixed $t \ge 1$, there are $\left(2^{2t-1 \choose t},2t-1\right)$ formations with no $\Up(3,t).$
\end{thm}

\begin{proof}
Let $S$ be the set of subsets of $\{1,2,...,2t-1\}$ of size $t$, so $|S|= {2t-1 \choose t}$. Define sets $S_1, S_2, ..., S_{2t-1}$ so that $S_i$ contains the subsets containing $i$. Then any $t$ sets $S_{i_1}, ... S_{i_t}$ will only agree on one element of $S$, the set $\{i_1,...,i_t\}$.

Identify the ${2t-1 \choose t}$ subsets with the numbers $\{0,1,...,{2t-1 \choose t}-1\}$. Let $n_i$ be $\sum_{j\in S_i} 2^j$. Then $[\tau_{n_1}, ..., \tau_{n_{2t-1}}]$ is a $\left(2^{2t-1\choose t},2t-1\right)$ formation with no $\Up(3,t)$ by Theorem~\ref{agree_binary}.
\end{proof}

The lower bound and upper bound meet, so these are the best possible for some particular values of $r$. For example, let $g={2t-1 \choose t}.$ We constructed a $\left(2^g,2t-1\right)$ formation with no $\Up(3,t)$ but every $\left(2^g+1,2t-1\right)$ formation contains an $\Up(3,t).$

\begin{thm}
For all fixed $t \ge 1$, there are $\left(\Omega(r^{2t-1 \choose t}),2t-1\right)$ formations with no $\up(r,t).$
\end{thm}
\begin{proof}
By theorem \ref{tensor} applied to the construction of theorem \ref{uprtconstruction}, there are $\left(2^{{2t-1 \choose t}k},2t-1\right)$-formations with no $\Up(2^k+1,t).$ 

To avoid an $\up(r,t)$, let $2^k$ be the greatest power of $2$ up to $\lceil \frac{r}{t^2-t+1}\rceil-1$. There is a $\left(2^{{2t-1 \choose t}k},2t-1\right)$ formation with no $\Up(2^k+1,t)$ hence no $\up((2^k+1)(t^2-t+1),t)$ hence no $\up(r,t).$ For a fixed $t$, $(2^k)^{2t-1 \choose t}$ is $\Omega\left(r^{2t-1 \choose t}\right).$
\end{proof}

\begin{cor}
For all fixed $t \ge 1$, we have $\fl(\up(r,t)) = \Theta(r^{2t-1 \choose t})$, where the constants in the bound depend only on $t$.
\end{cor}

\section{Sharp bounds using formation width}\label{fws}

The next theorem confirms a conjecture from \cite{gpt} that  $\fw(abc(acb)^t a b c) = 2t+3$ for all $t \geq 0$, where $\fw(u)$ denotes the minimum $s$ for which there exists $r$ such that every $(r, s)$-formation contains $u$. This gives an upper bound of $\Ex(a b c (a c b)^{t} a b c, n) = O(\F_{\fl(abc(acb)^t a b c), 2t+3}(n)) \leq n 2^{\frac{1}{t!}\alpha(n)^{t} + O(\alpha(n)^{t-1})}$ for $t \geq 1$ \cite{gpt, niv}. The sequence $abc(acb)^t a b c$ contains $(ab)^{t+2}$, so there was already a lower bound of  $\Ex(a b c (a c b)^{t} a b c, n)  = \Omega( \Ex((ab)^{t+2}, n)) \geq n 2^{\frac{1}{t!}\alpha(n)^{t} - O(\alpha(n)^{t-1})}$ for $t \geq 1$, where we are using Klazar's sparsity result \cite{klazar1} in the first inequality. The next result uses the fact proved in \cite{gpt} that $\fw(u)$ is the minimum $s$ for which every binary $(r, s)$-formation contains $u$, where $r$ is the number of distinct letters in $u$. Also in the next two results, we use the terminology $u$ \emph{has} $v$ to mean that some subsequence of $u$ is an exact copy of $v$, so $u$ \emph{has} $v$ is stronger than $u$ \emph{contains} $v$.

\begin{thm}
For all $t \geq 0$, $\fw(abc(acb)^t a b c) = 2t+3$. 
\end{thm}

\begin{proof}
The proof is trivial for $t = 0$, so suppose $t > 0$. It suffices by \cite{gpt} to show that every binary $(3, 2t+3)$-formation contains $u$. Consider any binary $(3, 2t+3)$-formation $f$ with permutations $x y z$ and $z y x$. Without loss of generality suppose permutations $3$ through $2t+1$ of $f$ have $(x y z)^{t}$. Then $f$ has $x z y (x y z)^{t} x z y$ unless the first six letters of $f$ are $z y x x y z$ or the last six letters of $f$ are $z y x x y z$. 

If the first six letters of $f$ are $z y x x y z$, then $f$ has $z y x (z x y)^{t} z y x$. So we assume that the last six letters of $f$ are $z y x x y z$. Now if the first six letters of $f$ are $z y x x y z$ or $z y x z y x$, then $f$ has $z y x (z x y)^t z y x$. Otherwise if the first six letters of $f$ are $x y z x y z$ or $x y z z y x$, then first note that $f$ has $x z y (x y z)^t x z y$ if the $3^{rd}$ through $(2t+1)^{st}$ permutations of $f$ have $(x y z)^{t+1}$. Otherwise the $3^{rd}$ through $(2t+1)^{st}$ permutations of $f$ have $(z y x)^{t-1}$, in which case $f$ has $x y z (x z y)^t x y z$.
\end{proof}

\begin{cor}
$\Ex(a b c (a c b)^{t} a b c, n) = n 2^{\frac{1}{t!}\alpha(n)^{t} \pm O(\alpha(n)^{t-1})}$ for $t \geq 1$.
\end{cor}

Next we prove that $\fw(abcacb(abc)^{t}acb) =2t+5$, which implies that $\Ex(a b c a c b (a b c)^{t} a c b, n) = n 2^{\frac{1}{(t+1)!}\alpha(n)^{t+1} \pm O(\alpha(n)^{t})}$ for $t \geq 1$.

\begin{lem}
For all $t \geq 0$, $\fw(abcacb(abc)^{t}acb) =2t+5$.
\end{lem}

\begin{proof}
First note that $abcacb(abc)^{t}acb$ has an alternation of length $2\mathrm{t}+6$, so $\fw(abcacb(abc)^{t}acb) \geq 2t+5$. To check the upper bound it suffices to show that every binary $(3,\ 2\mathrm{t}+5)$-formation contains $abcacb(abc)^{t}acb$ \cite{gpt}. We will denote an arbitrary binary $(3,\ 2\mathrm{t}+5)$-formation with permutations $x y z$ or $z y x$ by $p_{1}p_{2}\ldots p_{2t+5}.$ Without loss of generality, assume that $p_{5}p_{6}\ldots p_{2t+3}$ has $(xyz)^{t}$. If $p_{5}p_{6}\ldots p_{2t+3}$ has $(xyz)^{t+1}$ and $p_{1} =zyx$, then $p_{1}$ has $zx$, $p_{2}$ has $y$, $p_{3}p_{4}$ has $zy$, $p_{5}p_{6}\ldots p_{2t+3}$ has $xy(zxy)^{t}z$, and $p_{2t+4}p_{2t+5}$ has $yx$. Thus we have $zxyzyx(zxy)^{t}zyx$. If $p_{5}p_{6}\ldots p_{2t+3}$ has $(xyz)^{t+1}$ and $p_{1} = xyz$ then note that we can choose $xzy$ from $p_{2}p_{3}p_{4}$, $(xyz)^{t}xyz$ from $p_{5}p_{6}\ldots p_{2t+3}$, and $y$ from $p_{2t+4}$.

Now suppose that $p_{5}p_{6}\ldots p_{2t+3}$ has both $(xyz)^{t}$ and $(zyx)^{t-1}$. Note that if $p_{1}p_{2}p_{3}p_{4}$ has $xyzxzy$ and $p_{2t+4}p_{2t+5}$ has $xzy$ then $p_{1}p_{2}\ldots p_{2t+5}$ has $xyzxzy(xyz)^{t}xzy.$
It can be easily checked that $p_{1}p_{2}p_{3}p_{4}$ does not have $xyzxzy$ or $p_{2t+4}p_{2t+5}$ does not have $xzy$ exactly when $p_{2t+4}p_{2t+5} = zyxxyz$ or $p_{1}p_{2}p_{3}p_{4}$ $ =(zyx)(xyz)(zyx)(xyz), $ $(zyx)(zyx)(xyz)(zyx), $ $ (zyx)(zyx)(xyz)(xyz)$, or $(zyx)(zyx)(zyx)(xyz)$.

\noindent \textbf{Suppose that $p_{2t+4}p_{2t+5} = zyxxyz$ but $p_{1}p_{2}p_{3}p_{4}$ $ \not \in $ $ (zyx)(xyz)(zyx)(xyz), $ $(zyx)(zyx)(xyz)(zyx), $ $(zyx)(zyx)(xyz)(xyz),$ $(zyx)(zyx)(zyx)(xyz)$.}

If $p_{3} =zyx$ and $p_{1}p_{2} \neq xyzzyx$: Then $p_{1}p_{2}$ has {\it zxy}, $p_{3}$ has {\it zyx}, $p_{4}$ has the letter $z, p_{5}p_{6}\ldots p_{2t+3}$ has $xy(zxy)^{t-1}z$ and $p_{2t+4}p_{2t+5} = zyxxyz$. Thus we have $zxyzyx(zxy)^{t}zyx$.

If $p_{3} =xyz$ and $p_{1}p_{2} \neq zyxxyz$: Then $p_{1}p_{2}$ has $xzy$, $p_{3}$ has $xyz$, $p_{4}$ has the letter $x$, $p_{5}p_{6}\ldots p_{2t+3}$ has $zy(xzy)^{t-2}x$ and $p_{2t+4}p_{2t+5} = zyxxyz$. Thus we have $xzyxyz(xzy)^{t}xyz$.

If $p_{3} =zyx$ and $p_{1}p_{2} = xyzzyx$ and $p_{4} =xyz:$ Then $p_{1}p_{2}p_{3}p_{4}$ has $(xyz)(xzy)(xyz),$ $p_5 p_{6}\ldots p_{2t+3}$ has $(xyz)^{t-1}x$ and $p_{2t+4}p_{2t+5}= zyxxyz$. Thus we have $xyzxzy(xyz)^{t}xzy$.

If $p_{3} =zyx$ and $p_{1}p_{2} = xyzzyx$ and $p_{4} =zyx:$ Then $p_{1}p_{2}p_{3}p_{4}$ has $(xzy)(xyz)x,$ $p_5 p_{6}\ldots p_{2t+3}$ has $zy(xzy)^{t-2}x$ and $p_{2t+4}p_{2t+5}= zyxxyz$. Thus we have $xzyxyz(xzy)^{t}xyz$.

If $p_{3} =xyz$ and $p_{1}p_{2} = zyxxyz$ and $p_{4} =zyx:$ Then $p_{1}p_{2}p_{3}p_{4}$ has $(zyx)(zxy)(zyx),$ $p_5 p_{6}\ldots p_{2t+3}$ has $(zyx)^{t-1}$ and $p_{2t+4}p_{2t+5}=zyxxyz$. Thus we have $zyxzxy(zyx)^{t}zxy$.

If $p_{3} = xyz$ and $p_{1}p_{2} = zyxxyz$ and $p_{4} = xyz$: Then $p_{1}p_{2}p_{3}p_{4}$ has $(zxy)(zyx)z$, $p_5 p_{6}\ldots p_{2t+3}$ has $xy(zxy)^{t-1}z$ and $p_{2t+4}p_{2t+5}= zyxxyz$. Thus we have $zxyzyx(zxy)^{t}zyx$.

\noindent \textbf{Next, suppose that $p_{1}p_{2}p_{3}p_{4}\in (zyx)(xyz)(zyx)(xyz), (zyx)(zyx)(xyz)(zyx), (zyx)(zyx)(xyz)(xyz)$, $(zyx)(zyx)(zyx)(xyz)$.}

If $p_{1}p_{2}p_{3}p_{4} = (zyx)(zyx)(zyx)(xyz)$ or $p_{1}p_{2}p_{3}p_{4} = (zyx)(xyz)(zyx)(xyz)$: then $p_{1}p_{2}p_{3}p_{4}$ has $(zxy)(zyx)z$ and $p_{5}p_{6}\ldots p_{2t+3}$ has $xy(zxy)^{t-1}z$ and $p_{2t+4}p_{2t+5}$ has $yx$. Thus in this case we have $zxyzyx(zxy)^{t}zyx.$

If $p_{1}p_{2}p_{3}p_{4}=(zyx)(zyx)(xyz)(zyx)$ and $p_{2t+4}p_{2t+5}\neq xyzzyx$: then $p_{1}p_{2}p_{3}p_{4}$ has $(zyx)(zxy)(zyx)$ and $p_{5}p_{6}\ldots p_{2t+3}$ has $(zyx)^{t-1}$ and $p_{2t+4}p_{2t+5}$ has $zxy$. Thus in this case we have $zyxzxy(zyx)^{t}zxy.$

If $p_{1}p_{2}p_{3}p_{4}=(zyx)(zyx)(xyz)(zyx)$ and $p_{2t+4}p_{2t+5}=xyzzyx$: then $p_{1}p_{2}p_{3}p_{4}$ has $zxyzyx$ and $p_{5}p_{6}\ldots p_{2t+3}$ has $(zxy)^{t-1}z$ and $p_{2t+4}p_{2t+5}$ has $xyzyx$. Thus in this case we have $zxyzyx(zxy)^{t}zyx.$

If $p_{1}p_{2}p_{3}p_{4}=(zyx)(zyx)(xyz)(xyz)$ and $p_{5}=xyz$ and $p_{2t+4}p_{2t+5}\neq xyzzyx$: then $p_{1}p_{2}p_{3}p_{4}p_{5}$ has $(zyx)( zxy)( zyx)$ and $p_{6}p_{7}\ldots p_{2t+3}$ has $(zyx)^{t-1}$ and $p_{2t+4}p_{2t+5}$ has $zxy$. Thus in this case we have $zyxzxy(zyx)^{t}zxy.$

If $p_{1}p_{2}p_{3}p_{4}=(zyx)(zyx)(xyz)(xyz)$ and $p_{5}=xyz$ and $p_{2t+4}p_{2t+5}=xyzzyx$: then $p_{1}p_{2}p_{3}p_{4}p_{5}$ has $(zxy)(zyx)z$ and $p_{6}p_{7}\ldots p_{2t+3}$ has $(xyz)^{t-1}$ and $p_{2t+4}p_{2t+5}$ has $xyzyx$. Thus in this case we have $zxyzyx(zxy)^{t}zyx.$

If $p_{1}p_{2}p_{3}p_{4}=(zyx)(zyx)(xyz)(xyz)$ and $p_{5}=zyx$ and $p_{2t+4}p_{2t+5}\neq zyxxyz$: then $p_{1}p_{2}p_{3}p_{4}p_{5}$ has $xyzxzy$ and $p_{6}p_{7}\ldots p_{2t+3}$ has $(xyz)^{t}$ and $p_{2t+4}p_{2t+5}$ has $xzy$. Thus in this case we have $xyzxzy(xyz)^{t}xzy.$

If $p_{1}p_{2}p_{3}p_{4}=(zyx)(zyx)(xyz)(xyz)$ and $p_{5}=zyx$ and $p_{2t+4}p_{2t+5}=zyxxyz$: then $p_{1}p_{2}p_{3}p_{4}p_{5}$ has $(xzy)(xyz)(xzy)x$ and $p_{6}p_{7}\ldots p_{2t+3}$ has $zy(xzy)^{t-3}x$ and $p_{2t+4}p_{2t+5}$ has $zyxyz$. Thus in this case we have $xzyxyz(xzy)^{t}xyz.$
\end{proof}

\begin{cor}
$\Ex(a b c a c b (a b c)^{t} a c b, n) = n 2^{\frac{1}{(t+1)!}\alpha(n)^{t+1} \pm O(\alpha(n)^{t})}$ for $t \geq 1$.
\end{cor}

\section{Exact values}\label{exact}

Klazar \cite{klazar1}, Nivasch \cite{niv}, and Pettie \cite{pettie} showed that $\F_{r, 2}(n) < r n$, $\F_{r, 3}(n) < 2 r n$, $\F_{r, 2}(n, m) < n + (r-1)m$, and $\F_{r, 3}(n, m) < 2 n+(r-1)m$. In this section, we provide several elementary proofs to obtain exact values for all of the extremal functions in the previous sentence. In particular, we show that $\F_{r,2}(n, m) = n+(r-1)(m-1)$, $\F_{r,3}(n, m) = 2n+(r-1)(m-2)$, $\F_{r,2}(n) = (n-r)r+2r-1$, and $\F_{r,3}(n) = 2(n-r)r+3r-1$. We assume that $n \geq r$ for all of the results in this section.

\begin{thm}\label{exactfr}
For all integers $m \ge 1, n \ge r$ we have
\begin{enumerate}
\item \label{exactfr1} $\F_{r,2}(n, m) = n+(r-1)(m-1)$
\item \label{exactfr2} $\F_{r,3}(n, m) = 2n+(r-1)(m-2)$
\item $\F_{r,2}(n) = (n-r)r+2r-1$
\item $\F_{r,3}(n) = 2(n-r)r+3r-1$
\end{enumerate}
\end{thm}

\begin{proof}
We prove a matching upper bound and lower bound for each part.
\begin{enumerate}
\item Suppose that $u$ is a sequence on $m$ blocks with $n$ distinct letters that avoids $\mathcal{F}_{r,2}$. Delete the first occurrence of every letter in $u$. This empties the first block, leaving a sequence with at most $m-1$ nonempty blocks that must have at most $r-1$ letters per block, or else $u$ would have contained a pattern of $\mathcal{F}_{r,2}$. Thus $u$ has length at most $n+(r-1)(m-1)$, giving the upper bound.

For the lower bound, consider the sequence obtained from concatenating $\up(n, 1)$ with $\up(r-1,m-1)$. This sequence has $n$ distinct letters, $m$ blocks, and clearly avoids $\mathcal{F}_{r,2}$.

\item Suppose that $u$ is a sequence on $m$ blocks with $n$ distinct letters that avoids $\mathcal{F}_{r,3}$. Delete the first occurrence of every letter in $u$, as well as the last occurrence. This empties the first and last blocks, leaving a sequence with at most $m-2$ nonempty blocks that must have at most $r-1$ letters per block, or else $u$ would have contained a pattern of $\mathcal{F}_{r,3}$. Thus $u$ has length at most $2n+(r-1)(m-2)$, giving the upper bound.

For the lower bound, consider the sequence obtained from concatenating $\up(n, 1)$, $\up(r-1,m-2)$, and $\up(n, 1)$ again. This sequence has $n$ distinct letters, $m$ blocks, and clearly avoids $\mathcal{F}_{r,3}$.

\item Suppose that $u$ is an $r$-sparse sequence with $n$ distinct letters that avoids $\mathcal{F}_{r,2}$. Partition $u$ into blocks of size $r$, except for the last block which may have size at most $r$. Every block of length $r$ must have the first occurrence of some letter (or else $u$ would contain a pattern in $\mathcal{F}_{r,2}$), and the first block has $r$ first occurrences. This gives the upper bound.

For the lower bound, consider the sequence obtained by starting with $\up(r-1, 1)$ and concatenating $a_x \up(r-1, 1)$ to the end for $x = r, \dots, n$. This sequence has length $(n-r)r+2r-1$, it is $r$-sparse, and clearly avoids $\mathcal{F}_{r,2}$.

\item Suppose that $u$ is an $r$-sparse sequence with $n$ distinct letters that avoids $\mathcal{F}_{r,3}$. Partition $u$ into blocks of size $r$, except for some block besides the first or last which may have size at most $r$. Every block of length $r$ must have the first or last occurrence of some letter (or else $u$ would contain a pattern in  $\mathcal{F}_{r,3}$), the first block has $r$ first occurrences, and the last block has $r$ last occurrences. This gives the upper bound.

For the lower bound, consider the sequence obtained by starting with $\up(r-1, 1)$ and concatenating $a_x \up(r-1, 1)$ to the end for $x = r, \dots, n, r, \dots, n$. This sequence has length $2(n-r)r+3r-1$, it is $r$-sparse, and clearly avoids $\mathcal{F}_{r,3}$.
\end{enumerate}
\end{proof}

Next we find the exact value of $\Ex(\up(r,1) a_x, n, m)$ and $\Ex(\up(r,1) a_x, n)$ for $x \in \left\{1, \dots, r\right\}$. 

\begin{thm}
For all integers $m \ge 1, n \ge r$ we have
\begin{enumerate}
\item If $x \in \left\{1, \dots, r\right\}$, then $\Ex(\up(r,1) a_x, n, m) = n+(r-1)(m-1)$.
\item If $x \in \left\{1, \dots, r\right\}$, then $\Ex(\up(r,1) a_x, n) = n+x-1$.
\end{enumerate}
\end{thm}

\begin{proof}
As in the last result, we prove a matching upper bound and lower bound for each part.
\begin{enumerate}
\item The upper bound follows since every $(r,2)$-formation contains $\up(r,1) a_x$. For the lower bound, consider the sequence $u$ obtained from concatenating $\up(r-1,m-1)$ with $\up(n, 1)$. For any copy of $\up(r, 1)$ in $u$, any letter occurring after the copy must be making its first occurrence in $u$. Thus $u$ avoids $\up(r,1) a_x$, and it has $n$ distinct letters and $m$ blocks.

\item For the upper bound, let $u = u_1 u_2 \dots$ be an $r$-sparse sequence with $n$ distinct letters that avoids $\up(r,1) a_x$. Note that all of the letters $u_i$ for $i \geq x$ cannot occur later in $u$, or else $u$ would contain $\up(r,1) a_x$. This implies the upper bound.

For the lower bound, consider the sequence obtained by concatenating $\up(n, 1)$ with $\up(x-1, 1)$. Any letter that occurs twice in this sequence must have all occurrences among the first $x-1$ and last $x-1$ letters in the sequence. Thus this sequence avoids $\up(r,1) a_x$, it has length $n+x-1$, and it is $r$-sparse for $n \geq r$.
\end{enumerate}
\end{proof}

\section{Hypermatrices and generalized formations}\label{ddim}
In this section, we extend some of the exact results we proved for $(r, s)$-formations in Section \ref{exact} from sequences to $d$-dimensional 0-1 matrices. Before proving the results, we discuss some additional terminology.

For any family of $d$-dimensional 0-1 matrices $\mathcal{Q}$, define $\ex(n, \mathcal{Q}, d)$ to be the maximum number of ones in a $d$-dimensional matrix of sidelength $n$ that has no submatrix which can be changed to an exact copy of an element of $\mathcal{Q}$ by changing any number of ones to zeroes. When $\mathcal{Q}$ has only one element $Q$, we also write $\ex(n, \mathcal{Q}, d)$ as $\ex(n, Q, d)$. Most research on $\ex(n, \mathcal{Q}, d)$ has been on the case $d = 2$, but several results for $d = 2$ have been generalized to higher values of $d$. For example, Marcus and Tardos proved that $\ex(n, P, 2) = O(n)$ for every permutation matrix $P$ \cite{MT}, and this was later generalized by Klazar and Marcus \cite{KM}, who proved that $\ex(n, P, d) = O(n^{d-1})$ for every $d$-dimensional permutation matrix $P$. Another example is the upper bound $\ex(n, P, 2) = O(n)$ for double permutation matrices $P$ from \cite{gendbl}, which was generalized to an $O(n^{d-1})$ upper bound for $d$-dimensional double permutation matrices in \cite{gtm}.

Define the \emph{projection} $\overline P$ of the $d$-dimensional 0-1 matrix $P$ to be the $(d-1)$-dimensional 0-1 matrix with $\overline P(x_1, \dots, x_{d-1}) = 1$ if and only if there exists $y$ such that $P(y, x_1, \dots, x_{d-1}) = 1$. An \emph{$i$-row} of a $d$-dimensional 0-1 matrix is a maximal set of entries that have all coordinates the same except for the the $i^{th}$ coordinate. An \emph{$i$-cross section} of a $d$-dimensional 0-1 matrix is a maximal set of entries that have the same $i^{\text{th}}$ coordinate. 

Given a $d$-dimensional 0-1 matrix $P$ with $r$ ones, a $(P, s)$-formation is a $(d+1)$-dimensional 0-1 matrix $M$ with $s r$ ones that can be partitioned into $s$ disjoint $(d+1)$-dimensional 0-1 matrices $G_{1}, \dots, G_{s}$ each with $r$ ones so that any two $G_{i}, G_{j}$ have ones in the same sets of $1$-rows of $M$, the greatest first coordinate of any one in $G_{i}$ is less than the least first coordinate of any one in $G_{j}$ for $i < j$, and $P = \overline M$. For each $d$-dimensional $0-1$ matrix $P$, define $\mathcal{F}_{P, s}$ to be the set of all $(P,s)$-formations. Geneson \cite{ff01} proved that $ex(n,\mathcal{F}_{P,3},d+1) \leq 3(ex(n,P,d) n+ n^{d})$ for all positive integers $n$ and $d$-dimensional 0-1 matrices $P$. Here we prove that $ex(n,\mathcal{F}_{P,3},d+1) = 2n^d+\ex(n, P,d)(n-2)$ and $ex(n,\mathcal{F}_{P,2},d+1) = n^d+\ex(n, P,d)(n-1)$ for any $d$-dimensional 0-1 matrix $P$. By using a generalization of the K\H{o}v\'{a}ri-S\'{o}s-Tur\'{a}n upper bound, we also generalize a result from \cite{gcs} by proving that $\ex(n,\mathcal{F}_{P,s},d+1) =\Omega(n^{d+1-o(1)})$ if and only if $s = s(n) = \Omega(n^{1-o(1)})$.

Our first result in this section is an analogue of Theorem~\ref{exactfr}.\ref{exactfr1} for $d$-dimensional 0-1 matrices.

\begin{thm}
If $P$ is a $d$-dimensional 0-1 matrix, then $ex(n,\mathcal{F}_{P,2},d+1) = n^d+\ex(n, P,d)(n-1)$.
\end{thm}

\begin{proof}
Suppose that $A$ is a $(d+1)$-dimensional 0-1 matrix of sidelength $n$ that avoids $\mathcal{F}_{P,2}$. Delete the first one in every $1$-row. This empties the first $1$-cross section, leaving a $(d+1)$-dimensional 0-1 matrix with at most $n-1$ nonempty $1$-cross sections that must have at most $\ex(n, P,d)$ ones per $1$-cross section, or else $A$ would have contained an element of $\mathcal{F}_{P,2}$. Thus $A$ has at most $n^d+\ex(n, P,d)(n-1)$ ones, giving the upper bound.

For the lower bound, consider the $(d+1)$-dimensional 0-1 matrix obtained from concatenating a $d$-dimensional all-ones matrix with $n-1$ copies of a $d$-dimensional 0-1 matrix with $\ex(n, P, d)$ ones that avoids $P$. This matrix has $n^d+\ex(n, P,d)(n-1)$ ones and clearly avoids $\mathcal{F}_{P,2}$.
\end{proof}

The next equality improves on the bound $ex(n,\mathcal{F}_{P,3},d+1) \leq 3(ex(n,P,d) n+ n^{d})$ from \cite{ff01}. It is an analogue of Theorem~\ref{exactfr}.\ref{exactfr2} for $d$-dimensional 0-1 matrices.

\begin{thm}
If $P$ is a $d$-dimensional 0-1 matrix, then $ex(n,\mathcal{F}_{P,3},d+1) = 2n^d+\ex(n, P,d)(n-2)$.
\end{thm}

\begin{proof}
Suppose that $A$ is a $(d+1)$-dimensional 0-1 matrix of sidelength $n$ that avoids $\mathcal{F}_{P,3}$. Delete the first and last one in every $1$-row. This empties the first and last $1$-cross sections, leaving a $(d+1)$-dimensional 0-1 matrix with at most $n-2$ nonempty $1$-cross sections that must have at most $\ex(n, P,d)$ ones per $1$-cross section, or else $A$ would have contained an element of $\mathcal{F}_{P,3}$. Thus $A$ has at most $2n^d+\ex(n, P,d)(n-2)$ ones, giving the upper bound.

For the lower bound, consider the $(d+1)$-dimensional 0-1 matrix obtained from concatenating a $d$-dimensional all-ones matrix, $n-2$ copies of a $d$-dimensional 0-1 matrix with $\ex(n, P, d)$ ones that avoids $P$, and a $d$-dimensional all-ones matrix again. This matrix has $2n^d+\ex(n, P,d)(n-2)$ ones and clearly avoids $\mathcal{F}_{P,3}$.
\end{proof}

Wellman and Pettie asked how large must $s = s(n)$ be for $\Ex(a_s, n, n) = \Omega(n^{2-o(1)})$ \cite{petwel}. Geneson proved that $\Ex(a_s, n, n) = \Omega(n^{2-o(1)})$ if and only if $s = s(n) = \Omega(n^{1-o(1)})$ using the K\H{o}v\'{a}ri-S\'{o}s-Tur\'{a}n theorem \cite{gcs}. We extend this bifurcation result to formations in $d$-dimensional 0-1 matrices, but we need to use an extension of the K\H{o}v\'{a}ri-S\'{o}s-Tur\'{a}n theorem for $d$-dimensional 0-1 matrices. One such extension was proved in \cite{gtm}, where it was shown that $\ex(n,R^{k_1, \ldots , k_d},d)=O(n^{d- \alpha(k_1, \ldots , k_d)})$, where $\alpha = {\max({k_1, \ldots , k_d}) \over  k_1 \cdot k_2 \cdots k_d }$. This bound is not sufficient to extend the bifurcation result, but the same proof that was used in \cite{gtm} implies the following stronger result.

\begin{thm}
\label{upperbound}
For fixed $k_1, \dots, k_d$, $\ex(n,R^{j, k_1, \ldots , k_d},d+1)=O(j^{ {1 \over  k_1 \cdot k_2 \cdots k_d }}n^{d+1-  {1 \over  k_1 \cdot k_2 \cdots k_d }})$.
\end{thm}

Using Theorem \ref{upperbound}, we prove the following generalization of the result of Geneson \cite{gcs}.

\begin{thm}
If $P$ is a nonempty $d$-dimensional 0-1 matrix, then $\ex(n,\mathcal{F}_{P,s},d+1) =\Omega(n^{d+1-o(1)})$ if and only if $s(n) = \Omega(n^{1-o(1)})$.
\end{thm}

\begin{proof}
Suppose that $P$ is a nonempty $d$-dimensional 0-1 matrix with dimensions $k_1 \times \cdots \times k_d$. If $s = \Omega(n^{1-o(1)})$, then any $(d+1)$-dimensional 0-1 matrix that has $\min(s-1, n)$ $1$-cross sections with all entries equal to $1$ and $n-\min(s-1, n)$ $1$-cross sections with all entries equal to $0$ will avoid every $(P,s)$-formation. Thus in this case we have $\ex(n,\mathcal{F}_{P,s},d+1) \geq n^d (s-1) = \Omega(n^{d+1-o(1)})$.

If $s \neq \Omega(n^{1-o(1)})$, then there exists a constant $\alpha < 1$ and an infinite sequence of positive integers $i_1 < i_2 <\dots$ such that $s(i_j) < i_{j}^{\alpha}$ for each $j > 0$. Thus, it suffices to show that for every $0 < \alpha < 1$, there exists a constant $\beta < d+1$ such that $\ex(n,\mathcal{F}_{P,\ceil{n^{\alpha}}},d+1) = O(n^{\beta})$. However this follows immediately from Theorem \ref{upperbound}, since every $(d+1)$-dimensional 0-1 matrix that contains $R^{\ceil{n^{\alpha}},k_1, \ldots , k_d}$ must also contain an element of $\mathcal{F}_{P,\ceil{n^{\alpha}}}$.
\end{proof}

\section{Conclusion}

In this paper, we improved the upper bound on $\Ex(\up(r, 2), n)$ by showing that every $((r-1)\binom{r}{2}+1,3)$-formation contains $\up(r, 2)$. We proved that this result is sharp up to a constant factor by showing that there exist $(m, 3)$-formations with $m = \Omega(r^3)$ which avoid $\up(r, 2)$. More generally, we showed that $\fl(\up(r,t)) = \Theta(r^{2t-1 \choose t})$, where the constant in the bound depends only on $t$. 

Since these bounds are sharp but not exact, they leave some natural open problems. First, what is the exact value for the minimum $m = m(r)$ such that every $(m, 3)$-formation must contain $\up(r, 2)$? More generally, what is the exact value for the minimum $m = m(r, t)$ such that every $(m, 2t-1)$-formation contains $\up(r, t)$? 

It was shown in \cite{gpt} that $\fw(u) = 4$ and $\Ex(u, n) = \Theta(n \alpha(n))$ for any sequence $u$ of the form $a v a v' a$ such that $a$ is a letter, $v$ is a nonempty sequence of distinct letters excluding $a$, and $v'$ is obtained from $v$ by only moving the first letter of $v$. As in the last paragraph, there is the problem of determining the minimum $m = m(v)$ such that every $(m, 4)$-formation contains $a v a v' a$.

We determined the exact values of $\F_{r, 2}(n)$, $\F_{r, 3}(n)$, $\F_{r, 2}(n, m)$, and $\F_{r, 3}(n, m)$, and we also found the exact values of $\Ex(\up(r,1) a_x, n)$ and $\Ex(\up(r,1) a_x, n, m)$ for $x \in \left\{1, \dots, r\right\}$. A natural problem is to determine the exact value of $\Ex(\up(r,1) a_x a_y, n)$ and $\Ex(\up(r,1) a_x a_y, n, m)$ for any $x, y \in \left\{1, \dots, r\right\}$.

We also affirmed a conjecture from \cite{gpt} that  $\fw(abc(acb)^t a b c) = 2t+3$ for all $t \geq 0$ and that $\Ex(a b c (a c b)^{t} a b c, n) = n 2^{\frac{1}{t!}\alpha(n)^{t} \pm O(\alpha(n)^{t-1})}$ for $t \geq 1$. In order to determine other families of sequences $u$ for which $\fw(u)$ gives sharp upper bounds on $\Ex(u, n)$, it would be useful to have a faster algorithm for computing $\fw(u)$. The current fastest algorithms are in \cite{gtseq} and \cite{g_alg}. The latter algorithm is faster than the former when the number of distinct letters is fixed, as the length of the sequences goes to infinity. However the former algorithm is faster than the latter when the number of distinct letters approaches the length of the sequence.


\begin{thebibliography}{}
\bibitem{agshsh} P. Agarwal, M. Sharir, and P. Shor. Sharp upper and lower bounds on the length of general Davenport-Schinzel sequences. J. Combin. Theory Ser. A, 52 (1989) 228-274.
\bibitem{beame} P. Beame, E. Blais, and D. Huynh-Ngoc. Longest common subsequences in sets of permutations (2018) \url{https://arxiv.org/pdf/0904.1615.pdf}
\bibitem{chung} F. Chung, On unimodal subsequences. Journal of Combinatorial Theory Series A, 29 (1980) 267-279.
\bibitem{ck} J. Cibulka and J. Kyn\v{c}l. Tight bounds on the maximum size of a set of permutations with bounded vc-dimension. Journal of Combinatorial Theory Series A, 119 (2012) 1461-1478.
\bibitem{DS} H. Davenport and A. Schinzel. A combinatorial problem connected with differential equations. American J. Mathematics, 87 (1965) 684-69.
\bibitem{foxpachsuk} J. Fox, J. Pach, and A. Suk.  The number of edges in k-quasiplanar graphs. SIAM Journal of Discrete Mathematics 27 (2013) 550–561.
\bibitem{gen_tuple} J. Geneson, A Relationship Between Generalized Davenport-Schinzel Sequences and Interval Chains. Electr. J. Comb. 22 (2015): P3.19.
\bibitem{g_alg} J. Geneson, An algorithm for bounding extremal functions of forbidden sequences (2019) \url{https://arxiv.org/abs/1912.04897}
\bibitem{gcs} J. Geneson. Constructing sparse Davenport-Schinzel sequences. Discrete Mathematics 343 (2020) 111888.
\bibitem{gendbl} J.T. Geneson, Extremal functions of forbidden double permutation matrices, J. Combin. Theory Ser. A, 116 (2009) 1235-1244.
\bibitem{ff01} J. Geneson, Forbidden formations in multidimensional 0-1 matrices. Eur. J. Comb. 78 (2019) 147-154.
\bibitem{gpt} J. Geneson, R. Prasad, and J. Tidor, Bounding Sequence Extremal Functions with Formations. Electr. J. Comb. 21 (2014) P3.24.
\bibitem{gtm} J. Geneson, P. Tian, Extremal functions of forbidden multidimensional matrices. Discrete Mathematics 340 (2017) 2769-2781.
\bibitem{gtseq} J. Geneson and P. Tian, Sequences of formation width 4 and alternation length 5 (2015) \url{https://arxiv.org/abs/1502.04095}
\bibitem{keszegh} B. Keszegh, On linear forbidden submatrices, J. Combin. Theory Ser. A 116 (2009) 232-241.
\bibitem{klazar}  M. Klazar. Generalized Davenport-Schinzel sequences: results, problems, and applications. Integers 2 (2002) A11.
\bibitem{klazar1} M. Klazar. A general upper bound in the extremal theory of sequences. Commentationes Mathematicae Universitatis Carolinae, 33 (1992) 737-746.
\bibitem{KM} M. Klazar and A. Marcus, Extensions of the linear bound in the Furedi-Hajnal conjecture, Advances in Applied Mathematics, 38 (2007) 258-266.
\bibitem{kst} T. K\"{o}vari, V. T. S\'{o}s, and P. Tur\'{a}n. On a problem of K. Zarankiewicz. Colloquium Math., 3 (1954) 50-57.
\bibitem{MT} A. Marcus and G. Tardos, Excluded permutation matrices and the Stanley-Wilf conjecture, J. Combin. Theory Ser. A, 107 (2004) 153-160.
\bibitem{niv}  G. Nivasch. Improved bounds and new techniques for Davenport-Schinzel sequences and their generalizations. J. ACM, 57 (2010).
\bibitem{pettie} S. Pettie. Sharp bounds on Davenport-Schinzel sequences of every order. J. ACM, 62 (2015).
\bibitem{rs} D. Roselle and R. Stanton. Some properties of Davenport-Schinzel sequences. Acta Arithmetica, XVII:355-362, 1971.
\bibitem{agsh} M. Sharir and P. Agarwal. Davenport-Schinzel Sequences and their Geometric Applications. Cambridge University Press, 1995.
\bibitem{petwel} J. Wellman and S. Pettie, Lower bounds on Davenport-Schinzel sequences via rectangular Zarankiewicz matrices. Discrete Mathematics 341(7): 1987-1993, 2018.
\end{thebibliography}
\end{document}